\documentclass[12pt]{article}
\usepackage{mathrsfs}
\usepackage{epic,eepic,epsf,epsfig}
\usepackage{amsfonts,srcltx,mathrsfs}
\usepackage{multirow}
\usepackage{amsmath}
\usepackage{amssymb}
\usepackage{amsbsy}
\usepackage{graphicx}
\usepackage{amsfonts}
\usepackage{color}
\usepackage{setspace}

\newtheorem{lem}{Lemma}[section]%
\newtheorem{theorem}[lem]{Theorem}%
\newtheorem{cor}[lem]{Corollary}%

\def\nd{\mathrel{\bigm|\kern-.7em/}}

\def\f{\noindent}

\def\P\GammaL{\hbox{\rm P\GammaL}}

\begin{document}
\title{Spectral radius and edge-disjoint connected factors of graphs}

\footnotetext{* Corresponding author}
\footnotetext{E-mails: 13021531326@163.com; zhangwq@pku.edu.cn}

\author{Xinying Tang, Wenqian Zhang*\\
{\small School of Mathematics and Statistics, Shandong University of Technology}\\
{\small Zibo, Shandong 255000, P.R. China}}
\date{}
\maketitle

\begin{abstract}
For a graph $G$, the spectral radius of $G$ is the largest eigenvalue of its adjacency matrix. A connected factor of $G$ is a connected spanning subgraph of $G$. For example, a spanning tree of $G$ is a 1-connected factor of $G$. Let $G$ be a graph of order $n$ with minimum degree $\delta\geq6$, where $n\geq3\delta$. In this paper, we give a sharp spectral radius condition for $G$ to contain  $k$ edge-disjoint 2-connected factors and $\left\lfloor\frac{\delta-4k}{2}\right\rfloor$ edge-disjoint spanning trees, where $1\leq k\leq\left\lfloor\frac{\delta}{4}\right\rfloor$ is an integer.
\bigskip

\f {\bf Keywords:} spectral radius; eigenvalue; edge-disjoint spanning tree; connected factor; rigid graph.

\end{abstract}

\baselineskip 17 pt

\section{Introduction}

All graphs considered in this paper are simple graphs.
For a graph $G$, let $\overline{G}$ denote its complement graph. The vertex set and edge set of $G$ are denoted by $V(G)$ and $E(G)$ respectively. Set $e(G)=|E(G)|$. For $B\subseteq V(G)$, let $G[B]$ be the subgraph induced by $B$, and let $G-B=G[V(G)-B]$. Set $e_{G}(B)=|E(G[B])|$. For two disjoint subsets $U,V\subseteq V(G)$, let $e_{G}(U,V)$ denote the number of edges between $U$ and $V$ in $G$.  For a vertex $u\in V(G)$, let $d_{G}(u)$ denote the {\em degree} of $u$. Let $\delta(G)$ denote the minimum degree of $G$. For integer $n\geq1$, let $K_{n}$ be the complete graph of order $n$.  For two graphs $G_{1}$ and $G_{2}$, let $G_{1}\cup G_{2}$ denote the disjoint union of them, and let $G_{1}\vee G_{2}$ denote the join of them, i.e., the graph obtained from $G_{1}\cup G_{2}$ by adding all the edges between $V(G_{1})$ and $V(G_{2})$. 

Let $G$ be a graph with vertices $u_{1},u_{2},...,u_{n}$. The adjacency matrix $A(G)$ of $G$ is  $(a_{ij})_{n\times n}$, where $a_{ij}=1$ if $u_{i}$ is adjacent to $u_{j}$, and $a_{ij}=0$ otherwise. For any $1\leq i\leq n$, let $\lambda_{i}(G)$ denote the $i$-th largest eigenvalue of $A(G)$ (or $G$). The spectral
radius of $G$ is the largest eigenvalue $\lambda_{1}(G)$. By Perron-Frobenius Theorem, $\lambda_{1}(G)$ has non-negative eigenvectors (called Perron vectors). Moreover, $\lambda_{1}(G)$ is a simple root and has a positive eigenvector, when $G$ is connected. Let $D(G)$ denote the degree diagonal matrix of $G$. The Laplacian matrix of $G$ is defined as $L(G)=D(G)-A(G)$. The {\em algebraic connectivity} $\mu_{2}(G)$ of $G$  is the second smallest eigenvalue of $L(G)$. It is known that $G$ is connected if and only if $\mu_{2}(G)>0$.

Let $G$ be a connected graph. A {\em connected factor} of $G$ is a connected spanning subgraph. Clearly, a spanning tree of $G$ is a 1-connected factor. Let $\tau(G)$ denote the maximum number of edge-disjoint spanning trees
contained in  $G$.  Motivated by Kirchhoff's matrix tree theorem (see \cite{Kir})  and  a question of Seymour (see \cite{CW}), Cioab\u{a} and Wong \cite{CW} studied the following problem.

\medskip

\f{\bf Problem 1.} {\rm (\cite{CW})} Let $G$ be a connected graph. Determine the relationship between $\tau(G)$ and eigenvalues of $G$.

\medskip

The crucial tool for Problem 1 is the  famous Tree Packing Theorem due to  Tutte and Nash‐Williams (see \cite{Nash,T}). Cioab\u{a} and Wong \cite{CW} initiated the
study of the relationship between $\tau(G)$ and $\lambda_{2}(G)$ for $d$-regular graphs $G$. They \cite{CW} proved $\tau(G)\geq k$ if $\lambda_{2}(G)<d-\frac{2(2k-1)}{d+1}$, and conjectured that the condition can be released to $\lambda_{2}(G)<d-\frac{2k-1}{d+1}$. Note that this conjecture is essentially tight (see \cite{COPP}). This conjecture was extended in \cite{GLLY}, and partially confirmed in \cite{LS,LHL} and completely confirmed in\cite{LHGL}. 
Recently, the spectral radius (and other eigenvalue) conditions for edge-disjoint spanning trees of graphs were extensively studied (see \cite{CZ,DWL,FGL,GW,GW2,HWD,LLT,ZF}).

For the study of edge-disjoint  2-connected factors of graphs, we first introduce the concept of rigidity of graphs. Arising from mechanics, {\em rigidity} is the property of a structure that does not flex, and has been extensively studied in discrete
geometry and combinatorics, and has broad applications in material science, engineering and biological sciences (see \cite{CSC,GC,GSS,J1,ZFBG}). 
A $d$-dimensional bar-and-joint framework $(G,p)$ is the combination of a graph $G=(V(G),E(G))$ and a map $p: V(G)\rightarrow \mathbb{R}^{d}$ that assigns a point in $\mathbb{R}^{d}$ to each vertex of $G$. Two frameworks $(G,p)$ and $(G,q)$ are  {\em equivalent} (or {\em congruent}) if $\|p(u)-p(v)\|=\|q(u)-q(v)\|$ for any $uv\in E(G)$ (or for all pairs $u,v\in V(G)$), where $\|\cdot\|$ denote the Euclidean norm in $\mathbb{R}^{d}$. A framework $(G,p)$ is {\em generic} if the coordinates of its points are algebraically independent over $\mathbb{Q}$. The framework $(G,p)$ is {\em rigid} in $\mathbb{R}^{d}$ if there exists an $\epsilon>0$ such that every framework $(G,q)$ equivalent to $(G,p)$ satisfying $\|p(u)-q(u)\|<\epsilon$ for any $u\in V(G)$ is congruent to $(G,p)$.
Asimov and Roth \cite{AR} proved that there is a generic framework $(G, p)$ which is rigid in $\mathbb{R}^{d}$ if and only if every generic framework of $G$ is rigid in $\mathbb{R}^{d}$.
Thus, the generic rigidity is inherent properties of a graph. A graph $G$ called rigid in $\mathbb{R}^{d}$, if there is a generic framework of $G$ which is rigid in $\mathbb{R}^{d}$.

In this paper, we only consider rigid graphs in $\mathbb{R}^{2}$. In 1970, Laman \cite{Laman} gave a combinatorial characterization for these graphs. on rigidity of graphs . From Laman's characterization, it is not hard to see that every rigid graph in $\mathbb{R}^{2}$ with at least $3$ vertices is $2$-connected. The connectivity conditions for edge-disjoint rigid spanning subgraphs in graphs were extensively studied (see \cite{CDGS,G,G1,JJ1,J,LY}).

Cioab\u{a}, Dewar and  Gu \cite{CDG} gave algebraic connectivity conditions for edge-disjoint rigid spanning subgraphs in graphs.

\begin{theorem}{\rm (\cite{CDG})}\label{algebraic condi}
Let $G$ be a graph with minimum degree $\delta(G)\geq6k$. If\\
$(1)$ $\mu_{2}(G)>\frac{6k-1}{\delta(G)+1}$,\\
$(2)$ $\mu_{2}(G-u)>\frac{4k-1}{\delta(G-u)+1}$ for every $u\in V(G)$, and\\
$(3)$ $\mu_{2}(G-v-w)>\frac{2k-1}{\delta(G-v-w)+1}$ for every $v,w\in V(G)$,\\
then $G$ contains at least $k$ edge-disjoint spanning rigid subgraphs.
\end{theorem}

The following corollary is directly deduced from Theorem \ref{algebraic condi}.

\begin{cor}{\rm (\cite{CDG})}\label{algebraic cor}
Let $G$ be a graph with minimum degree $\delta\geq6k$. If
$\mu_{2}(G)>2+\frac{2k-1}{\delta(G)-1}$,
then $G$ contains at least $k$ edge-disjoint spanning rigid subgraphs.
\end{cor}

In this paper, we will gave a sharp spectral radius condition for edge-disjoint spanning rigid subgraphs and trees in graphs. Note that the sharp spectral radius  condition for  2-connected graphs to be rigid was given by Fan, Huang and Lin \cite{FHL}. As the main result of this paper, we proved the following Theorem \ref{main}.

\medskip

For integers $k\geq1,\delta\geq4k$ and $n\geq3\delta$, let $G_{n,\delta,k}$ be the graph obtained from $K_{1}\vee(K_{\delta}\cup K_{n-1-\delta})$ by connecting one vertex in $V(K_{\delta})$ to $(k-1)$ vertices in $V(K_{n-1-\delta})$.

\begin{theorem}\label{main}
Assume that $\delta\geq6$ and $n\geq3\delta$ are two integers. Let $G$ be a graph of order $n$ with minimum
degree $\delta$. For any integer $k$ with $1\leq k\leq\frac{\delta}{4}$, if $\lambda_{1}(G)\geq\lambda_{1}(G_{n,\delta,k})$, then $G$ contains
edge-disjoint $k$ spanning rigid subgraphs and $\left\lfloor\frac{\delta-4k}{2}\right\rfloor$ spanning trees unless $G=G_{n,\delta,k}$.
\end{theorem}

Now we point out that the extremal graph $G_{n,\delta,k}$ in Theorem \ref{main} does not contain $k$ edge-disjoint spanning rigid subgraphs. In fact, since a rigid graph $G$ is $2$-connected, we see that $G-u$ is connected (and thus contains a spanning tree) for any $u\in V(G)$. Suppose that  $G_{n,\delta,k}$ contains $k$ edge-disjoint spanning rigid subgraphs. Then the graph obtained from $G_{n,\delta,k}$ by deleting the vertex in $V(K_{1})$ contains $k$ edge-disjoint spanning trees. However, this is not possible from the construction of $G_{n,\delta,k}$.

 The rest of this paper is organized as follows. In Section 2, we present some spectral tools. In section 3, we will prove a useful lemma.  In Section 4, we first introduce some applications of matroid theory in the study of edge-disjoint spanning subgraphs, and then give the proof of Theorem \ref{main}.

\section{Spectral tools}

In this section, we will introduce some spectral tools. The first Lemma is taken from Theorem 8.1.1 of \cite{CRS}, and the second one is a variation (with a
very similar proof) of Theorem 8.1.3 of \cite{CRS}. 

\begin{lem} {\em (\cite{CRS})}\label{subgraph}
If $H$ is a subgraph of a connected graph $G$, then $\lambda_{1}(H)\leq\lambda_{1}(G)$, with equality if and only if $H=G$.
\end{lem}

\begin{lem} {\em (\cite{CRS})}\label{trans}
Let $G$ be a connected graph with a Perron vector $\mathbf{x}=(x_{z})_{z\in V(G)}$. Assume that $uv$ is an edge and $u'v'$ is not an edge of $G$.
Let $G'$ be
the graph obtained from $G$ by deleting the edge $uv$ and adding the non-edge
 $u'v'$. If $x_{u}x_{v}\leq x_{u'}x_{v'}$, then $\lambda_{1}(G')>\lambda_{1}(G)$.
\end{lem}

The following lemma can be found in \cite{HSF,N}.

 \begin{lem} {\em (\cite{HSF,N})}\label{inequality}
 Let $G$ be a graph of order $n$ with $m$ edges and minimum degree $\delta$. Then
 $\lambda_{1}(G)\leq\frac{\delta-1}{2}+\sqrt{2m-\delta n+\frac{(\delta+1)^{2}}{4}}$.
\end{lem}

\section{A useful lemma}

For integers $k\geq1,\delta\geq4k$ and $n\geq3\delta$, let $\mathcal{G}_{n,\delta,k}$ be the set of graphs obtained from $K_{1}\vee(K_{\delta}\cup K_{n-1-\delta})$ by adding $(k-1)$ edges between $V(K_{\delta})$ and $V(K_{n-1-\delta})$.
Recall that $G_{n,\delta,k}$ (defined in Section 1) is the graph obtained from $K_{1}\vee(K_{\delta}\cup K_{n-1-\delta})$ by connecting one vertex in $V(K_{\delta})$ to  $(k-1)$ vertices in $V(K_{n-1-\delta})$. Clearly, $G_{n,\delta,k}\in\mathcal{G}_{n,\delta,k}$.

\begin{lem} \label{set-extremal}
For integers $k\geq1,\delta\geq4k$ and $n\geq3\delta$, let \( \mathcal{G}_{n,\delta,k} \) be defined as above. Then \( G_{n,\delta,k} \) is the unique extremal graph with the maximum spectral radius among the graphs in \( \mathcal{G}_{n,\delta,k} \).
\end{lem}

\f{\bf Proof:} Let \( G \) be an extremal graph with the maximum spectral radius among the graphs in \( \mathcal{G}_{n,\delta,k} \). It suffices to show that \( G = G_{n,\delta,k} \). 
The conclusion holds trivially for \( k = 1, 2 \). Thus we can assume \( k \geq 3 \) in the following discussion. 

Recall that \( G \) can be obtained from \( K_1 \vee (K_\delta \cup K_{n-1-\delta}) \) by adding \( (k-1) \) edges between \( V(K_\delta) \) and \( V(K_{n-1-\delta}) \). Denote \( V(K_1) = \{w\} \), \( V(K_\delta) = \{u_1, u_2, \ldots, u_\delta\} \) and \( V(K_{n-1-\delta}) = \{v_1, v_2, \ldots, v_{n-1-\delta}\} \). Let \( x = (x_u)_{u \in V(G)} \) be a Perron vector of \( G \). Without loss of generality, assume that \( x_{u_1} \geq x_{u_2} \geq \cdots \geq x_{u_\delta} \) and \( x_{v_1} \geq x_{v_2} \geq \cdots \geq x_{v_{n-1-\delta}} \).
Let \( s \geq 0 \) be the largest integer such that \( v_s \) is adjacent to \( u_1 \) in \( G \). 

\medskip

\f{\bf Claim 1.}  \( u_1 \) is adjacent to \( v_i \) for any \( 1 \leq i \leq s \) in \( G \). Moreover,  \( u_i v_\ell \) is not an edge of \( G \) for any \( 1\leq i \leq\delta \) and \( \ell >s \).

\medskip

\f{\bf Proof of Claim 1.} It suffices to show that if $u_{i}v_{j}$ is an edge of $G$, then $u_{i'}v_{j'}$ is also an edge for any $i'\leq i$ and $j'\leq j$. Suppose not. Let  \( G_1 \) be the graph obtained from \( G \) by deleting the edge $u_{i}v_{j}$ and adding the non-edge $u_{i'}v_{j'}$. Clearly, \( G_1 \) is in \( \mathcal{G}_{n,\delta,k} \). Since $x_{u_{i}}x_{v_{j}} \leq x_{u_{i'}}x_{v_{j'}}$, we have \( \lambda_1(G_1) > \lambda_1(G) \) by Lemma \ref{trans}. But this contradicts the choice of \( G \). This finishes the proof of Claim 1. \hfill$\Box$

\medskip

Denote \( \lambda_1(G) = \lambda \). By symmetry, we have \( x_{v_{s+1}} = x_v \) for any \( s + 2 \leq i \leq n - 1 - \delta \) from Claim 1. By \( A(G)x = \lambda x \), we have
\[
\lambda x_{v_{s+1}} = x_w + x_{v_1} + \left( \sum_{2 \leq i \leq s} x_{v_i} \right) + (n - 2 - \delta - s)x_{v_{s+1}} \geq x_w + x_{v_1} + (n - 3 - \delta)x_{v_{s+1}}.
\]
It follows that 
$$ x_{v_{s+1}} \geq \frac{x_w + x_{v_1}}{\lambda + 3 + \delta - n}.$$

Let \( G' \) be the graph obtained from \( K_1 \vee (K_\delta \cup K_{n-1-\delta}) \) by connecting the vertex \( u_1 \) to the vertices \( v_1, v_2, \ldots, v_{k-1} \). Clearly, \( G' = G_{n,\delta,k} \). It suffices to prove that \( G = G' \). If \( s = k - 1 \), then \( G = G' \), as desired. Now assume that \( 1 \leq s < k - 1 \). We will obtain a contradiction. Denote  \( \lambda_1(G') = \lambda' \). Let \( y = (y_u)_{u \in V(G')} \) be a Perron vector of \( G' \). By symmetry, we have \( y_{u_2} = y_{u_j} \) for any \( 3 \leq j \leq \delta \) and \( y_{v_1} = y_{v_j} \) for any \( 2 \leq j \leq k - 1 \).  By \( A(G')y = \lambda'y \), we have
\[
\lambda'y_{u_2} = y_w + y_{u_1} + (\delta - 2)y_{u_2}.
\]
It follows that 
$$y_{u_2} = \frac{y_w + y_{u_1}}{\lambda' + 2 - \delta}.$$

\medskip

\f{\bf Claim 2.}  \( \frac{y_w}{y_{u_1}} < \frac{n-1-\delta}{k} \).

\medskip

\f{\bf Proof of Claim 2.} By \( A(G')y = \lambda'y \), we have
\[
\lambda'y_w \leq (n - 1 - \delta)y_{v_1} + y_{u_1} + (\delta - 1)y_{u_2}
\]
and
\[
\lambda'y_{u_1} = y_w + (k - 1)y_{v_1} + (\delta - 1)y_{u_2} \geq ky_{v_1} + (\delta - 1)y_{u_2}.
\]
Thus,
\[
\frac{y_w}{y_{u_1}} = \frac{\lambda'y_w}{\lambda'y_{u_1}} \leq \frac{(n - 1 - \delta)y_{v_1} + y_{u_1} + (\delta - 1)y_{u_2}}{ky_{v_1} + (\delta - 1)y_{u_2}}.
\]
Clearly, 
$$\lambda'y_{u_1} \leq ky_w + (\delta - 1)y_{u_2}$$ 
and
\[
\lambda'y_{u_2} = y_w + y_{u_1} + (\delta - 2)y_{u_2} \geq y_w + (\delta - 1)y_{u_2}.
\]
Thus,
\[
\frac{y_{u_1}}{y_{u_2}} = \frac{\lambda'y_{u_1}}{\lambda'y_{u_2}} \leq \frac{ky_w + (\delta - 1)y_{u_2}}{y_w + (\delta - 1)y_{u_2}} \leq k.
\]
It follows that 
$$\frac{y_{u_1} + (\delta - 1)y_{u_2}}{(\delta - 1)y_{u_2}} \leq \frac{k+\delta-1}{\delta-1}.$$
Since \( k \leq \frac{\delta}{4} \) and \( n \geq 3\delta \), we have
$$\frac{(n-1-\delta)y_{v_1}}{ky_{v_1}} = \frac{n-1-\delta}{k} > \frac{k+\delta-1}{\delta-1}
\geq\frac{y_{u_1} + (\delta - 1)y_{u_2}}{(\delta - 1)y_{u_2}}.$$ 
Then
\[
\frac{y_w}{y_{u_1}} \leq \frac{(n - 1 - \delta)y_{v_1} + y_{u_1} + (\delta - 1)y_{u_2}}{ky_{v_1} + (\delta - 1)y_{u_2}}<\frac{(n-1-\delta)y_{v_1}}{ky_{v_1}} = \frac{n - 1 - \delta}{k},
\]
as desired. \hfill$\Box$

\medskip

By Lemma \ref{subgraph}, we have \( \lambda' > n - 1 - \delta \), since \( G' \) contains a copy of \( K_{n-\delta} \). Clearly, 
$$2e(G) = (n - \delta)(n - \delta - 1) + \delta(\delta - 1) + 2(\delta + k - 1).$$
 By Lemma \ref{inequality}, noting \( n \geq 3\delta \), it is easy to check that
\[
\lambda = \lambda_1(G) \leq \frac{\delta - 1}{2} + \sqrt{2e(G) - n\delta + \left( \frac{\delta + 1}{4} \right)^2} < n - \delta.
\]
Clearly, if \( u_i v_j \) is an edge of \( G \) and is not an edge of \( G' \), then \( 2 \leq i \leq \delta \) and \( 1 \leq j \leq s \).  
Recall that \( \frac{y_w}{y_{u_1}} < \frac{n-1-\delta}{k} \) by Claim 2. Trivially, \( \frac{x_w}{x_{v_1}} \geq 1 \). Note that 
$$\frac{2}{3} >\frac{\frac{n-1-\delta}{k}}{n-2\delta+1}$$
 as \( k \geq 3 \) and \( n \geq 3\delta \).  
Thus,  
\[
(\lambda' - \lambda)x^T y = x^T (A(G') - A(G))y
\]
\[
= \left( \sum_{u_1 v_i \in (E(G') - E(G))} (x_{u_1} y_{v_i} + x_{v_i} y_{u_1}) \right) - \sum_{u_i v_j \in (E(G) - E(G'))} (x_{u_i} y_{v_j} + x_{v_j} y_{u_i})
\]
\[
\geq (k-1-s)(x_{u_1} y_{v_1} + x_{v_{s+1}} y_{u_1} - x_{u_2} y_{v_1} - x_{v_1} y_{u_2})
\]
\[
\geq (k-1-s)(x_{v_{s+1}} y_{u_1} - x_{v_1} y_{u_2})
\]
\[
\geq (k-1-s)\left( \frac{x_w + x_{v_1}}{\lambda + 3 + \delta - n} y_{u_1} - \frac{y_w + y_{u_1}}{\lambda' + 2 - \delta} x_{v_1} \right)
\]
\[
= (k-1-s)x_{v_1} y_{u_1} \left( \frac{\frac{x_w}{x_{v_1}} + 1}{\lambda + 3 + \delta - n} - \frac{\frac{y_w}{y_{u_1}} + 1}{\lambda' + 2 - \delta} \right)
\]
\[
> (k-1-s)x_{v_1} y_{u_1} \left( \frac{1+1}{3} - \frac{\frac{n-1-\delta}{k}}{n-2\delta+1} \right)
> 0.\]
That is, \( (\lambda' - \lambda)x^T y > 0 \), implying that \( \lambda' > \lambda \). But this contradicts the choice of \( G \). So, the proof is completed. \hfill$\Box$

\section{Proof of Theorem \ref{main}}

To prove Theorem \ref{main}, we first introduce a result from matroid theory. For basic definitions on  matroid, one may refer to the book \cite{O}.

Let $G$ be a graph with edge set $E=E(G)$. For any $F\subseteq E$, let $G(F)$ denote the spanning subgraph of $G$ with edge set $F$. 
Let $$\mathcal{I}=\left\{F\subseteq E~|~G(F)~ is~a~forest\right\}.$$
 The {\em circuit matroid} $\mathcal{M}(G)$ of $G$ is the matroid on the ground set $E(G)$ with independent set $\mathcal{I}$. The rank function of $\mathcal{M}(G)$ is given by $$r_{\mathcal{M}}(F)=|V(G)|-c(F),$$
  where $F\subseteq E$ and $c(F)$ denotes the number of components of $G(F)$.  Clearly, $G$ contains a spanning tree if and only if $r_{\mathcal{M}}(E)=|V(G)|-1$.

Let $F$ be a non-empty subset of $E$, and let $S$ be the set of vertices
incident with some edge in $F$. Set $H=G(F)$. The edge set $F$ is called
{\em independent} in $G$, if $e(H[X])\leq2|X|-3$ for all $X\subseteq V(G)$ with $|X|\geq2$. 
Let $$\mathcal{I}=\left\{F\subseteq E~|~F~ is~independent~in~G\right\}.$$
 Jackson and Jord\'{a}n \cite{JJ} proved that $\mathcal{R}(G)=(E,\mathcal{I})$ is a matroid on the ground set $E$ with independent set $\mathcal{I}$, which is called the {\em rigid matroid} of $G$.   The rank function of $\mathcal{R}(G)$ is given by  
 $$r_{\mathcal{R}}(F)=\min\left\{\sum_{1\leq i\leq t}(2|X_{i}|-3)\right\},$$
 where the minimum is taken over all collections of subsets $X_{1},X_{2},...,X_{t}$ of $V(G)$ such that $\left\{E(G[X_{1}]),E(G[X_{2}]),...,E(G[X_{t}])\right\}$ partitions $F$. (Clearly, $|X_{i}|\geq2$ for any $1\leq i\leq t$.)
$G$ is called rigid if $r_{\mathcal{R}}(E)=2|V(G)|-3$. Using a result of Laman \cite{Laman}, Jackson and Jord\'{a}n \cite{JJ} stated that these definitions agree with the geometric definitions for rigidity in $\mathbb{R}^{2}$ given in the introduction.

Let $G$ be a graph of order $n$ with edge set $E$. As in \cite{CDG}, let $\mathcal{N}_{k,\ell}$ denote the matroid on the ground set $E$ obtained by taking the matroid union of $k$ copies of the rigid matroids $\mathcal{R}(G)$ and $\ell$ copies of circuit matroids $\mathcal{M}(G)$. For $E'\subseteq E$, let $r_{k,\ell}(E')$ be the rank of $E'$. By a result of Edmonds on the rank of matroid union (see \cite{E} or Chapter 42 of \cite{S} for details), it can be obtained that
$$r_{k,\ell}(E)=\min_{F\subseteq E}\left\{k\cdot r_{\mathcal{R}}(F)+\ell\cdot r_{\mathcal{M}}(F)+|E-F|\right\}.$$
It is known that $r_{k,\ell}(E)\leq k(2n-3)+\ell(n-1)$, with equality if and only if $G$ contains edge-disjoint $k$ spanning rigid subgraphs and $\ell$ spanning trees. Thus, $G$ contains edge-disjoint $k$ spanning rigid subgraphs and $\ell$ spanning trees if and only if
$$\min_{F\subseteq E}\left\{k\cdot r_{\mathcal{R}}(F)+\ell\cdot r_{\mathcal{M}}(F)+|E-F|\right\}=k(2n-3)+\ell(n-1)~~~(*1).$$

Now we are ready to prove Theorem \ref{main}. (The formula $(*1)$ is very important in the proof of Theorem \ref{main}.)

\medskip

\f{\bf Proof of Theorem \ref{main}.}   Assume that $\lambda_{1}(G)\geq\lambda_{1}(G_{n,\delta,k})$ and $G$ is not $G_{n,\delta,k}$. We shall
prove that G contains edge-disjoint $k$ spanning rigid subgraphs and $\left\lfloor\frac{\delta-4k}{2}\right\rfloor$ spanning trees.
By Lemma \ref{subgraph}, we have \( \lambda_{1}(G)\geq\lambda_{1}(G_{n,\delta,k}) > n - 1 - \delta \), since \( G_{n,\delta,k}\) contains a copy of \( K_{n-\delta} \). 
 By Lemma \ref{inequality}, we have that
$$\lambda_1(G) \leq \frac{\delta - 1}{2} + \sqrt{2e(G) - n\delta + \left( \frac{\delta + 1}{4} \right)^2}.$$
Combining this with $\lambda_{1}(G)>n-1-\delta$, we can obtain that 
$$2e(G)>n^{2}-(2\delta+1)n+2\delta^{2}+\delta.$$
Then
 $$2e(\overline{G})<2\delta n-2\delta^{2}-\delta.$$

\medskip

\f{\bf Claim 1.} $G$ is $\delta$-edge-connected and $G-u$ is $k$-edge-connected for any vertex $u$ of $G$.

\medskip

\f{\bf Proof of Claim 1.} We  first show that $G$ is $\delta$-edge-connected. If $G$ is not $\delta$-edge-connected, then there is a $D\subseteq V(G)$, such that $e_{G}(D,V(G)-D)\leq\delta-1$. It is easy to show $|D|\geq\delta+1$ and $|V(G)-D|\geq\delta+1$. Then, noting $n\geq3\delta$,
\begin{equation}
\begin{aligned}
2e(\overline{G})&\geq\left(\sum_{v\in D}d_{\overline{G}}(v)\right)+\sum_{v\in (V(G)-D)}d_{\overline{G}}(v)\\
&\geq (n-|D|)|D|-(\delta-1)+|D|(n-|D|)-(\delta-1)\\
&\geq(\delta+1)(n-1-\delta)-(\delta-1)+(\delta+1)(n-1-\delta)-(\delta-1)\\
&=2(\delta+1)(n-1-\delta)-2(\delta-1)\\
&>2\delta n-2\delta^{2}-\delta.
\end{aligned}\notag
\end{equation}
This contradicts the fact that $2e(\overline{G})<2\delta n-2\delta^{2}-\delta$. Thus $G$ is $\delta$-edge-connected.

Now we show the second part of Claim 1. Suppose that there is a vertex $u$ of $G$ such that $G-u$ is not $k$-edge-connected. Let $V'=V(G-u)$. Then there is an $O\subseteq V'$, such that $e_{G-u}(O,V'-O)\leq k-1$. It is easy to show  $|O|\geq\delta$ and $|V'-O|\geq\delta$ as $\delta(G-u)\geq\delta-1$. 
If $|O|\geq\delta+1$ and $|V'-O|\geq\delta+1$. Then, noting $n\geq3\delta$ and $1\leq k\leq\frac{\delta}{4}$,
\begin{equation}
\begin{aligned}
2e(\overline{G})&\geq\left(\sum_{v\in O}d_{\overline{G}}(v)\right)+\sum_{v\in (V'-O)}d_{\overline{G}}(v)\\
&\geq (n-1-|O|)|O|-(k-1)+|O|(n-1-|O|)-(k-1)\\
&=2(\delta+1)(n-2-\delta)-2(k-1)\\
&\geq2\delta n-2\delta^{2}-\delta.
\end{aligned}\notag
\end{equation}
This contradicts the fact that $2e(\overline{G})<2\delta n-2\delta^{2}-\delta$. Thus, without loss of generality, we can assume that $|O|=\delta$. Let $G_{1}$ be the graph obtained from $G$ by adding edges such that $d_{G_{1}}(u)=n-1,e_{G-u}(O,V'-O)= k-1$ and both $G_{1}[O]$ and $G_{1}[V'-O]$ are complete graphs. Clearly, $G_{1}$ is in  $\mathcal{G}_{n,\delta,k}$.  By Lemma \ref{set-extremal}, we have $\lambda_{1}(G_{1})\leq\lambda_{1}(G_{n,\delta,k})$, with equality if and only if $G_{1}=G_{n,\delta,k}$. Note that $\lambda_{1}(G)\leq\lambda_{1}(G_{1})$ as $G$ is a subgraph of $G_{1}$. Consequently, we must have $G=G_{n,\delta,k}$ as $\lambda_{1}(G)\geq\lambda_{1}(G_{n,\delta,k})$, a contradiction. Thus $G-u$ is $k$-edge-connected for any vertex $u$ of $G$. This finishes the proof of Claim 1. \hfill$\Box$

\medskip

Now we prove the theorem by contradiction. Suppose that $G$ does not contain edge-disjoint $k$ spanning rigid subgraphs and $\left\lfloor\frac{\delta-4k}{2}\right\rfloor$ spanning trees. Set $E=E(G)$. Then by $(*1)$, there is a $F\subseteq E$ such that
$$k\cdot r_{\mathcal{R}}(F)+\left\lfloor\frac{\delta-4k}{2}\right\rfloor\cdot r_{\mathcal{M}}(F)+|E-F|<k(2n-3)+\left\lfloor\frac{\delta-4k}{2}\right\rfloor(n-1)~~~(*2).$$
By the definition of rank function, there is a collection $\mathcal{B}=\left\{B_{1},B_{2},...,B_{h}\right\}$ of subsets of $V(G)$ such that $\left\{E(G[B_{1}]),E(G[B_{2}]),...,E(G[B_{h}])\right\}$ partitions $F$, and 
$$ r_{\mathcal{R}}(F)=\sum_{B\in\mathcal{B}}(2|B|-3).$$
Recall that $r_{\mathcal{M}}(F)=n-c(F)$, where $c(F)$ denotes the number of components of $G(F)$.
Then $(*2)$ becomes
$$k\left(\sum_{B\in\mathcal{B}}(2|B|-3)\right)+\left\lfloor\frac{\delta-4k}{2}\right\rfloor (n-c(F))+|E-F|<k(2n-3)+\left\lfloor\frac{\delta-4k}{2}\right\rfloor(n-1).$$
It follows that
$$k\left(\sum_{B\in\mathcal{B}}(2|B|-3)\right)+|E-F|<k(2n-3)+
\left\lfloor\frac{\delta-4k}{2}\right\rfloor(c(F)-1)~~~(*3).$$
We can choose such an $F$ with minimum size $|F|$ that  $(*3)$ holds. 

\medskip

\f{\bf Claim 2.} $|B|\geq3$ for any $B$ of $\mathcal{B}$.

\medskip

\f{\bf Proof of Claim 2.} Suppose that $|B_{i}|=2$ for some $1\leq i\leq h$. Let $F'=F-\left\{e_{i}\right\}$ and $\mathcal{B}'=\mathcal{B}-\left\{B_{i}\right\}$, where $e_{i}$ is the (unique) edge contained in $G[B_{i}]$. Then the edges inside the elements of $\mathcal{B}'$ partition $F'$. Clearly, $c(F')\geq c(F)$. Then by $(*3)$, noting $k\geq1$ and $2|B_{i}|-3=1$, we have
$$k\left(\sum_{B\in\mathcal{B}'}(2|B|-3)\right)+
|E-F'|\leq k\left(\sum_{B\in\mathcal{B}}(2|B|-3)\right)+|E-F|$$
$$<k(2n-3)+
\left\lfloor\frac{\delta-4k}{2}\right\rfloor(c(F)-1)\leq
k(2n-3)+\left\lfloor\frac{\delta-4k}{2}\right\rfloor(c(F')-1).$$
Thus, $F'$ satisfies $(*3)$. But this contradicts the choice of $F$. Hence, we proved that $|B|\geq3$ for any $B$ of $\mathcal{B}$. \hfill$\Box$ 

\medskip

We call each $B$ of $\mathcal{B}$ a stone.

\medskip

Let $H=G(F)$ and $c=c(F)$. Let $Q_{1},Q_{2},...,Q_{c_{0}}$ be the non-singleton components of $H$, where $0\leq c_{0}\leq c$. 
Thus $(*3)$ becomes
$$k\left(\left(\sum_{B\in\mathcal{B}}(4|B|-6)\right)-4n\right)+2|E-F|
<-6k+2\left\lfloor\frac{\delta-4k}{2}\right\rfloor(c-1)\leq-6k+(\delta-4k)(c-1)~~(*4).$$
Clearly, each non-singleton component of $H$ consists of some stones in $\mathcal{B}$ and the edges inside them. Let $V_{0}$ be the set of isolated vertices of $H$, and let $V_{1}=V(G)-V_{0}$. Then $c_{0}+|V_{0}|=c$ and $V_{1}=\cup_{B\in\mathcal{B}}B$.

For any $u\in V_{1}$, let $r_{u}$ be the number of stones in $\mathcal{B}$ including the vertex $u$. Let $W=\left\{u\in V_{1}~|~r_{u}\geq2\right\}$ and $L=\left\{u\in V_{1}~|~r_{u}=1\right\}$. Then $V_{1}=W\cup L$. For any $u\in V(G)$, let $a_{u}$ be the number of edges in $E-F$ incident with the vertex $u$. For any $B$ in $\mathcal{B}$, 
let $$a_{B\cap L}=\sum_{u\in B\cap L}a_{u}.$$
 For any $0\leq j\leq c_{0}$, let $$b_{j}=\sum_{u\in V(Q_{j})\cap W}a_{u}.$$
 Set $n_{1}=|V_{1}|$. Then $n_{1}=n-(c-c_{0})$. For any $1\leq j\leq c_{0}$, 
 let $$I(Q_{j})=b_{j}+\sum_{B\subseteq V(Q_{j})}(a_{B\cap L}+2k|B\cap W|-6k).$$ 
 
 \medskip

\f{\bf Claim 3.} $\sum_{1\leq j\leq c_{0}}I(Q_{j})<(\delta-4k)(c_{0}-1)-6k$.

\medskip

\f{\bf Proof of Claim 3.}  Since
$$\sum_{B\in\mathcal{B}}\sum_{u\in B}\frac{1}{r_{u}}=\sum_{u\in V_{1}}\sum_{B\ni u,B\in\mathcal{B}}\frac{1}{r_{u}}=\sum_{u\in V_{1}}1=n_{1},$$
we have
\begin{equation}
\begin{aligned}
\left(\sum_{B\in\mathcal{B}}(4|B|-6)\right)-4n_{1}&=\left(\sum_{B\in\mathcal{B}}(4|B|-6)\right)
-4\sum_{B\in\mathcal{B}}\sum_{u\in B}\frac{1}{r_{u}}\\
&=\sum_{B\in\mathcal{B}}\left(4|B|-\left(\sum_{u\in B}\frac{4}{r_{u}}\right)-6\right)\\
&=\sum_{B\in\mathcal{B}}\left(\left(\sum_{u\in B}4(1-\frac{1}{r_{u}})\right)-6\right)\\
&\geq\sum_{B\in\mathcal{B}}\left(\left(\sum_{u\in B\cap W}4(1-\frac{1}{2})\right)-6\right)\\
&=\sum_{B\in\mathcal{B}}(2|B\cap W|-6)\\
&=\sum_{1\leq j\leq c_{0}}\sum_{B\subseteq V(Q_{j})}(2|B\cap W|-6)~~~(*5).
\end{aligned}\notag
\end{equation}
Note that $a_{u}=d_{G}(u)$ for any $u\in V_{0}$. Clearly,
$$2|E-F|=\left(\sum_{u\in V_{0}}a_{u}\right)+\sum_{1\leq j\leq c_{0}}\sum_{u\in V(Q_{j})}a_{u}=\left(\sum_{u\in V_{0}}a_{u}\right)+\sum_{1\leq j\leq c_{0}}\left(b_{j}+\sum_{B\subseteq V(Q_{j})}a_{B\cap L}\right)~~~(*6).$$
By $(*4),(*5)$ and $(*6)$, noting $a_{u}\geq\delta$ for any $u\in V_{0}$, we have
\begin{equation}
\begin{aligned}
&-6k+(\delta-4k)(c-1)\\
&>-4k(c-c_{0})+k\left(\sum_{1\leq j\leq c_{0}}\sum_{B\subseteq V(Q_{j})}(2|B\cap W|-6)\right)+\left(\sum_{u\in V_{0}}a_{u}\right)+\sum_{1\leq j\leq c_{0}}\left(b_{j}+\sum_{B\subseteq V(Q_{j})}a_{B\cap L}\right)\\
&\geq(c-c_{0})(\delta-4k)+\sum_{1\leq j\leq c_{0}}\left(b_{j}+\sum_{B\subseteq V(Q_{j})}\left(a_{B\cap L}+2k|B\cap W|-6k\right)\right)\\
&=(c-c_{0})(\delta-4k)+\sum_{1\leq j\leq c_{0}}I(Q_{j}).
\end{aligned}\notag
\end{equation}
It follows that
 $$\sum_{1\leq j\leq c_{0}}I(Q_{j})<(\delta-4k)(c_{0}-1)-6k.$$
This finishes the proof of Claim 3. \hfill$\Box$
 
 \medskip
 
By Claim 3, we see $c_{0}\geq1$, otherwise $0<-6k$.
A stone $B$ in $\mathcal{B}$ is called {\em loose}, if $|B\cap W|\leq2$ and $a_{B\cap L}+2k|B\cap W|-6k<\delta-4k$.

\medskip

\f{\bf Claim 4.} Let $B$ be a loose stone. Then we have the following conclusions.\\
$(1)$ If $|B\cap W|=0$, then $|B|\geq\delta$ and $a_{B\cap L}<\delta+2k$.\\
$(2)$ If $|B\cap W|=1$, then $|B|\geq\delta$ and $a_{B\cap L}<\delta$.\\
$(3)$ If $|B\cap W|=2$, then $|B|\geq\delta+1$ and $a_{B\cap L}<\delta-2k$.

\medskip

\f{\bf Proof of Claim 4.} We first show that $|B|\geq\delta$. Suppose that $|B|=p<\delta$. Recall that $p\geq3$. Let $|B\cap W|=x$, where $x\leq2$. Then $|B\cap L|=p-x$. Since the minimum degree of $G$ is $\delta$, we have $a_{B\cap L}\geq(\delta+1-p)(p-x)$. Since $\delta\geq4k,0\leq x\leq2$ and $3\leq p\leq \delta-1$, it is easy to check that 
$$(\delta+1-p)(p-x)+2kx\geq\delta-2+4k.$$ 
 Thus, noting  $k\geq1$,
  $$a_{B\cap L}+2k|B\cap W|-6k=(\delta+1-p)(p-x)+2kx-6k\geq\delta-2-2k\geq\delta-4k.$$
   This contradicts the fact that $B$ is loose. Hence we proved $|B|\geq\delta$.
   
Now $(1)$ and $(2)$ clearly hold. To show $(3)$, assume that $|B|=\delta$ by contradiction. Note that $|B\cap W|=2$ in this case. Clearly, $a_{B\cap L}\geq\delta-2$.
Thus, 
$$a_{B\cap L}+2k|B\cap W|-6k\geq\delta-2+4k-6k\geq\delta-2-2k\geq\delta-4k.$$
 This contradicts the fact that $B$ is loose. Thus, $|B|\geq\delta+1$. It holds clearly that $a_{B\cap L}<\delta-2k$ as $B$ is loose. This finishes the proof of Claim 4. \hfill$\Box$

\medskip

\f{\bf Claim 5.} There are at most two loose stones.

\medskip

\f{\bf Proof of Claim 5.} Suppose that there are at least three loose stones, say $B_{1},B_{2},B_{3}$. Then $|B_{i}\cap L|\geq\delta-1$ for any $1\leq i\leq 3$ by Claim 4. Let $B$ be any loose stone. Without loss of generality, we can assume that $B\neq B_{1},B_{2}$. Then $n-|B|\geq |B_{1}\cap L|+|B_{2}\cap L|\geq2(\delta-1)$. It follows that $|B|\leq n-\delta$. Now we show that
$$\sum_{u\in B\cap L}d_{\overline{G}}(u)\geq(\delta-1)(n-\delta-1)-(\delta-2k).$$
  If $|B\cap W|=0$, by $(1)$ of Claim 4, we have $|B|\geq\delta$, and
\begin{equation}
\begin{aligned}
\sum_{u\in B\cap L}d_{\overline{G}}(u)&\geq(n-|B|)|B\cap L|-a_{B\cap L}\\
&=(n-|B|)|B|-a_{B\cap L}\\
&\geq\delta(n-\delta)-(\delta+2k)\\
&\geq (\delta-1)(n-\delta-1)-(\delta-2k),
\end{aligned}\notag
\end{equation}
as desired.
  If $|B\cap W|=1$, by $(2)$ of Claim 4, we have
  $$\sum_{u\in B\cap L}d_{\overline{G}}(u)\geq(n-|B|)(|B|-1)-a_{B\cap L}\geq(\delta-1)(n-\delta)-\delta\geq(\delta-1)(n-\delta-1)-(\delta-2k),$$
  as desired.
  If $|B\cap W|=2$, by $(3)$ of Claim 4, we have $|B|\geq\delta+1$, and
 $$\sum_{u\in B\cap L}d_{\overline{G}}(u)\geq(n-|B|)(|B|-2)-a_{B\cap L}\geq(\delta-1)(n-\delta-1)-(\delta-2k),$$
 as desired.
Hence, we proved $$\sum_{u\in B_{i}\cap L}d_{\overline{G}}(u)\geq(\delta-1)(n-\delta-1)-(\delta-2k)$$
for any $1\leq i\leq3$.
Then, noting $n\geq3\delta$ and $\delta\geq6$, we have
\begin{equation}
\begin{aligned}
2e(\overline{G})&\geq\sum_{1\leq i\leq3}\sum_{u\in B_{i}\cap L}d_{\overline{G}}(u)\\
&\geq3(\delta-1)(n-\delta-1)-3(\delta-2k)\\
&>2\delta n-2\delta^{2}-\delta.
\end{aligned}\notag
\end{equation}
This contradicts the fact that $2e(\overline{G})<2\delta n-2\delta^{2}-\delta$.
 This finishes the proof of Claim 5. \hfill$\Box$

\medskip

 Now we have the following two case.
 
 \medskip
 
\f{\bf Case 1}. $c=1$.

\medskip

Then $c_{0}=1$ as $1\leq c_{0}\leq c=1$.
By Claim 3, we have $I(Q_{1})<-6k$. If $Q_{1}$ consists of exactly one stone, say $B_{1}$, then $I(Q_{1})=b_{1}+(a_{B_{1}\cap L}+2k|B_{1}\cap W|-6k)\geq-6k$, a contradiction. Hence, $Q_{1}$ consists of at least two stones. For each stone $B$ of $Q_{1}$, clearly $|B\cap W|\geq1$ as $c=1$. Recall 
$$I(Q_{1})=b_{1}+\sum_{B\subseteq V(Q_{1})}(a_{B\cap L}+2k|B\cap W|-6k).$$
 Clearly, $a_{B\cap L}+2k|B\cap W|-6k\geq0$ if $B$ is not loose. Now let $B$ be a loose stone. If $|B\cap W|=2$, then $a_{B\cap L}+2k|B\cap W|-6k\geq-2k$. If $|B\cap W|=1$, then $a_{B\cap L}\geq k$ since $G-u$ is $k$-edge-connected for any $u\in V(G)$ by Claim 1. It follows that $a_{B\cap L}+2k|B\cap W|-6k\geq-3k$.
Since there are at most two loose stones by Claim 5,  we have
$$I(Q_{1})=b_{1}+\sum_{B\subseteq V(Q_{1})}(a_{B\cap L}+2k|B\cap W|-6k)\geq-3k-3k=-6k,$$
 a contradiction.

\medskip

\f{\bf Case 2}. $c\geq2$.

\medskip

\f{\bf Claim 6.} For any given $1\leq \ell\leq c_{0}$, recall that $Q_{\ell}$ is a non-singleton component of $H$. Then we have the following conclusions.\\
$(1)$ $I(Q_{\ell})\geq\delta-4k$ if $Q_{\ell}$ contains no loose stone. \\
$(2)$ $I(Q_{\ell})\geq\delta-7k$ if $Q_{\ell}$ contains exactly one loose stone. \\
$(3)$ $I(Q_{\ell})\geq-6k$ if $Q_{\ell}$ contains exactly two loose stones.

\medskip

\f{\bf Proof of Claim 6.} $(1)$ Assume that $Q_{\ell}$ contains no loose stone. Then $|B\cap W|\geq3$ or $a_{B\cap L}+2k|B\cap W|-6k\geq\delta-4k$ for any stone $B$ of $Q_{\ell}$. In either case, we have 
$$a_{B\cap L}+2k|B\cap W|-6k\geq0.$$
 First consider that there is at least one stone of $Q_{\ell}$, say $B_{1}$, such that $a_{B_{1}\cap L}+2k|B_{1}\cap W|-6k\geq\delta-4k$. Then
\begin{equation}
\begin{aligned}
I(Q_{\ell})&=b_{\ell}+\sum_{B\subseteq V(Q_{\ell})}(a_{B\cap L}+2k|B\cap W|-6k)\\
&\geq(a_{B_{1}\cap L}+2k|B_{1}\cap W|-6k)+\sum_{B\subseteq V(Q_{\ell}),B\neq B_{1}}(a_{B\cap L}+2k|B\cap W|-6k)\\
&\geq\delta-4k+0\\
&=\delta-4k,
\end{aligned}\notag
\end{equation}
 as desired.

Now it remains to consider that  $a_{B\cap L}+2k|B\cap W|-6k<\delta-4k$ for any stone $B$ of $Q_{\ell}$. Then  $|B\cap W|\geq3$ for each stone $B$ of $Q_{\ell}$, since  $Q_{\ell}$ contains no loose stone. Since $G$ is $\delta$-edge-connected by Claim 1, we have $b_{\ell}+\sum_{B\subseteq V(Q_{\ell})}a_{B\cap L}\geq\delta$. Then
$$I(Q_{\ell})=b_{\ell}+\sum_{B\subseteq V(Q_{\ell})}(a_{B\cap L}+2k|B\cap W|-6k)\geq b_{\ell}+\sum_{B\subseteq V(Q_{\ell})}a_{B\cap L}\geq\delta,$$ as desired.

\medskip

$(2)$ Assume that $Q_{\ell}$ contains exactly one loose stone. Let $B_{1},B_{2},...,B_{t}$ be the stones of $Q_{\ell}$. Without loss of generality, assume that $B_{1}$ is the only one loose stone of $Q_{\ell}$. If $t=1$, then $|B_{1}\cap W|=0$ and $a_{B_{1}\cap L}\geq\delta$ as $G$ is $\delta$-edge-connected.
Thus $$I(Q_{\ell})=b_{\ell}+\sum_{B\subseteq V(Q_{\ell})}(a_{B\cap L}+2k|B\cap W|-6k)=b_{\ell}+(a_{B_{1}\cap L}+2k|B_{1}\cap W|-6k)\geq\delta-6k,$$
 as desired.
 Now let $t\geq2$. Then $|B_{i}\cap W|\geq3$ or $a_{B_{i}\cap L}+2k|B_{i}\cap W|-6k\geq\delta-4k$ for any $2\leq i\leq t$. In either case, we have $a_{B_{i}\cap L}+2k|B_{i}\cap W|-6k\geq0$ for any $2\leq i\leq t$.

Suppose that there is at least one non-loose stone of $Q_{\ell}$, say $B_{2}$, such that $$a_{B_{2}\cap L}+2k|B_{2}\cap W|-6k\geq\delta-4k.$$
 Note that $|B_{1}\cap W|\geq1$. Moreover, $a_{B_{1}\cap L}\geq k$ if $|B_{1}\cap W|=1$, since $G-u$ is $k$-edge-connected for any $u\in V(G)$ by Claim 1. It follows that $a_{B_{1}\cap L}+2k|B_{1}\cap W|-6k\geq-3k$.
 Then
 \begin{equation}
\begin{aligned}
I(Q_{\ell})&=b_{\ell}+\sum_{B\subseteq V(Q_{\ell})}(a_{B\cap L}+2k|B\cap W|-6k)\\
&\geq(a_{B_{1}\cap L}+2k|B_{1}\cap W|-6k)+\sum_{B\subseteq V(Q_{\ell}),B\neq B_{1}}(a_{B\cap L}+2k|B\cap W|-6k)\\
&\geq-3k+(\delta-4k)\\
&=\delta-7k,
\end{aligned}\notag
\end{equation}
 as desired.

Now it remains to consider that  $|B_{i}\cap W|\geq3$ for each $2\leq i\leq t$. Since $G$ is $\delta$-edge-connected, we have $b_{\ell}+\sum_{B\subseteq V(Q_{\ell})}a_{B\cap L}\geq\delta$.
 Then
 \begin{equation}
\begin{aligned}
I(Q_{\ell})&=b_{\ell}+\sum_{B\subseteq V(Q_{\ell})}(a_{B\cap L}+2k|B\cap W|-6k)\\
&=(2k|B_{1}\cap W|-6k)+\left(\sum_{2\leq i\leq t}(2k|B_{i}\cap W|-6k)\right)+\left(b_{\ell}+\sum_{B\subseteq V(Q_{\ell})}a_{B\cap L}\right)\\
&\geq -4k+0+\delta\\
&=\delta-4k,
\end{aligned}\notag
\end{equation}
as desired.

\medskip

$(3)$ Assume that $Q_{\ell}$ contains exactly two loose stones. Let $B_{1},B_{2},...,B_{t}$ be the stones of $Q_{\ell}$. Without loss of generality, assume that $B_{1}$ and $B_{2}$ are the only two loose stones of $Q_{\ell}$. Note that  $|B_{i}\cap W|\geq1$ for $1\leq i\leq2$. Thus $a_{B_{i}\cap L}+2k|B_{i}\cap W|-6k\geq-3k$, since $G-u$ is $k$-edge-connected for any $u\in V(G)$. Similar to (2),  we have $a_{B_{i}\cap L}+2k|B_{i}\cap W|-6k\geq0$ for any $3\leq i\leq t$.
Then
 \begin{equation}
\begin{aligned}
I(Q_{\ell})&=b_{\ell}+\left(\sum_{1\leq i\leq 2}(a_{B_{i}\cap L}+2k|B_{i}\cap W|-6k)\right)+\left(\sum_{3\leq i\leq t}(a_{B_{i}\cap L}+2k|B_{i}\cap W|-6k)\right)\\
&\geq0+2(-3k)+0\\
&=-6k,
\end{aligned}\notag
\end{equation}
 as desired.
This finishes the proof of Claim 6. \hfill$\Box$

\medskip

Recall that 
$$\sum_{1\leq j\leq c_{0}}I(Q_{j})<(\delta-4k)(c_{0}-1)-6k$$
 by Claim 3. Moreover, there are at most two loose stones by Claim 5. We still have the following 3 subcases.

\medskip

\f{\bf Subcase 2.1.} There are exactly two loose stones, say $B_{1}$ and $B_{2}$.

\medskip

Suppose that $B_{1}$ and $B_{2}$ are contained in one component of $H$, say $Q_{1}$. Then $I(Q_{1})\geq-6k$ by $(3)$ of Claim 6, and $I(Q_{i})\geq\delta-4k$ for any $2\leq i\leq c_{0}$ by $(1)$ of Claim 6. Thus
 $$\sum_{1\leq j\leq c_{0}}I(Q_{j})=I(Q_{1})+\sum_{2\leq j\leq c_{0}}I(Q_{j})\geq-6k+(\delta-4k)(c_{0}-1),$$
  a contradiction to Claim 3.

 It remains that $B_{1}$ and $B_{2}$ are contained in different components of $H$, say $Q_{1}$ and $Q_{2}$. Then $I(Q_{i})\geq\delta-7k$ for $1\leq i\leq2$ by $(2)$ of Claim 6, and $I(Q_{i})\geq\delta-4k$ for any $3\leq i\leq c_{0}$ by $(1)$ of Claim 6.
 Thus, noting $\delta\geq4k$, we have
 \begin{equation}
\begin{aligned}
\sum_{1\leq j\leq c_{0}}I(Q_{j})&=\left(\sum_{1\leq j\leq 2}I(Q_{j})\right)+\left(\sum_{3\leq j\leq c_{0}}I(Q_{j})\right)\\
&\geq2(\delta-7k)+(\delta-4k)(c_{0}-2)\\
&=\delta-10k+(\delta-4k)(c_{0}-1)\\
&\geq-6k+(\delta-4k)(c_{0}-1),
\end{aligned}\notag
\end{equation}
 a contradiction.

\medskip

\f{\bf Subcase 2.2.} There is exactly one loose stone of $H$, say $B_{1}$.

\medskip

Without loss of generality, assume that $B_{1}$ is contained in $Q_{1}$. Then $I(Q_{1})\geq\delta-7k$ by $(2)$ of Claim 6, and $I(Q_{i})\geq\delta-4k$ for any $2\leq i\leq c_{0}$ by $(1)$ of Claim 6. Thus
  \begin{equation}
\begin{aligned}
\sum_{1\leq j\leq c_{0}}I(Q_{j})&=I(Q_{1})+\sum_{2\leq j\leq c_{0}}I(Q_{j})\\
&\geq(\delta-7k)+(\delta-4k)(c_{0}-1)\\
&>-6k+(\delta-4k)(c_{0}-1),
\end{aligned}\notag
\end{equation}
a contradiction.

\medskip

\f{\bf Subcase 2.3.} There is no loose stone.

\medskip

By $(1)$ of Claim 6, we have $I(Q_{i})\geq\delta-4k$ for any $1\leq i\leq c_{0}$. Then
  $$\sum_{1\leq j\leq c_{0}}I(Q_{j})\geq(\delta-4k)c_{0}>-6k+(\delta-4k)(c_{0}-1),$$
   a contradiction.
This completes the proof. \hfill$\Box$

\medskip

\medskip

\medskip

\f{\bf Data availability statement}

\medskip

There is no associated data.

\medskip

\f{\bf Declaration of Interest Statement}

\medskip

There is no conflict of interest.

\end{document}